\documentclass[12pt]{article}

\usepackage{amsmath,amsfonts,amssymb,amsthm,amsrefs}

\title{Partitions of trees and $\aca^\prime$}
\author{Bernard A. Anderson\\ Jeffry L. Hirst\\ Appalachian State University}

\theoremstyle{plain}
\newtheorem{thm}{Theorem}
\newtheorem{lemma}[thm]{Lemma}

\newcommand{\nat}{\mathbb N}  
\newcommand{\aca}{{\sf{ACA}}_0}
\newcommand{\acap}{\aca^\prime}
\newcommand{\rca}{{\sf{RCA}}_0}
\newcommand{\halt}[2]{\Phi^X_{#1,#2}(m)\hskip -.3em\downarrow}
\newcommand{\tre}{2^{<\nat}}
\newcommand{\cat}{^\smallfrown}
\newcommand{\ctt}{{\sf{TT}}}
\newcommand{\lh}{\text{lh}}

\date{ January 27, 2009} 
\begin{document}

\maketitle

\begin{abstract}
We show that a version of Ramsey's theorem for trees for arbitrary exponents is
equivalent to the subsystem $\acap$ of reverse mathematics.
\end{abstract}

In \cite{chm}, a version of Ramsey's theorem for trees is analyzed using techniques
from computability theory and reverse mathematics.  In particular, it is shown that
for each standard integer $n \ge 3$, the usual Ramsey's theorem for $n$-tuples
is equivalent to the tree version for $n$-tuples.  The main result of this note shows
that the universally quantified versions of these forms of Ramsey's theorem are also equivalent.
Because there are so few examples of proofs involving $\acap$ in the literature,
we have included a somewhat detailed exposition of the proof.

The main subsystems of second order arithmetic used in this paper are $\rca$,
which includes a comprehension axiom for computable sets, and $\aca$, which
appends a comprehension axiom for sets definable by arithmetical formulas.
For details on the axiomatization of these subsystems, see \cite{simpson}.
More about the subsystem $\acap$ appears below.

If $\tre$ is the full binary tree of height $\omega$, we may identify each node
with a finite sequence of zeros and ones.  We refer to any subset of the nodes as a
subtree, and say that a subtree $S$ is isomorphic to $\tre$ if every node of $S$ has exactly two
immediate successors in $S$.  Formally, $S \subseteq \tre$ is isomorphic to $\tre$ if 
and only if there is a bijection $b: \tre \to S$ such that for all $\sigma, \tau \in \tre$, we have
$\sigma \subseteq \tau$ if and only if $b(\sigma ) \subseteq b ( \tau )$.  (For sequences,
$\sigma \subseteq \tau$ means $\sigma$ is an initial segment of $\tau$, and $\sigma \subset \tau$
means $\sigma$ is a proper initial segment of $\tau$.)  For any subtree $S$, we write
$[S]^n$ for the set of linearly ordered $n$-tuples of nodes in $S$.  All the nodes in
any such $n$-tuple are pairwise comparable in the tree ordering.  In \cite{chm}, the following
version of Ramsey's theorem is presented.
\begin{list}{}{}
\item  $\ctt (n)$:  Fix $k \in \nat$.  Suppose that $[\tre]^n$ is colored with $k$ colors.
Then there is a subtree $S$ isomorphic to $\tre$ such that $[S]^n$ is monochromatic.
\end{list}
In applying $\ctt (n)$, we often think of the coloring as a function $f: [\tre]^n \to k$, in which
case $S$ is monochromatic precisely when $f$ is constant on $[S]^n$.

Let $\halt{e}{t}$ denote a fixed formalization of the assertion that the Turing
machine with code number $e$, using an oracle for the set $X$, halts on input
$m$ with the entire computation bounded by $t$.  We will assume that $t$ is
a bound on all aspects of the computation, including codes for inputs from the
oracle.  This formalization can be based on Kleene's $T$-predicate or any
similar arithmetization of computation.
In $\rca$, we use the notation $Y \le_T X$ to denote the existence of two
codes $e$ and $e^\prime$ such that
\[
\forall m (m \in Y \leftrightarrow \exists t \halt{e}{t})
\]
and
\[
\forall m (m \notin Y \leftrightarrow \exists t \halt{e^\prime}{t}) .
\]
The preceding formalizes the notion that $Y$ is Turing reducible to $X$ if and
only if both $Y$ and its complement are computably enumerable in $X$.

As in \cite{dh}, we can also use this notation to formalize $\acap$.  Given any set $X$, let
$Y= X^\prime$ denote the statement
\[
\forall \langle m, e \rangle ( \langle m, e \rangle \in Y \leftrightarrow \exists t \halt {e}{t} ) ,
\]
where $\langle m,e \rangle$ denotes an integer code for the ordered pair $(m,e)$.  To formalize
the $n$th jump for $n\ge 1$, we write $Y = X^{(n)}$ if there is a finite sequence
$ X_0 , \dots , X_n$ such that $X_0 = X$, $X_n = Y$, and for
every $i<n$, $X_{i+1} = X_i^\prime$.  In this notation, $Y = X^\prime$ if and only if $Y = X^{(1)}$,
and we will often write $X^{\prime\prime}$ for $X^{(2)}$.
The subsystem $\acap$ consists of $\aca$ plus the assertion that for every $X$ and every $n$,
there is a set $Y$ such that $Y = X^{(n)}$.

Using all this terminology, we can prove a formalized version of the implication from $\ctt (n)$
to $\ctt (n+1)$, including a formalized computability theoretic upper bound.
\begin{lemma}\label{lemmaA}
$(\rca)$
Suppose $R$ is a tree isomorphic to $\tre$,
$f : [ R]^{n+1} \to k$ is a finite coloring of the $(n+1)$-tuples of comparable nodes of $R$,
and both $R \le_T A$ and $f \le_T A$.  Suppose that $A^{\prime\prime}$ exists.  Then
we can find a tree $S$ and a coloring $g: [S]^n \to k$ such that $S \le_T A^{\prime\prime}$,
$g \le_T A^{\prime\prime}$, $S$ is a subtree of $R$ isomorphic to $\tre$, and
every monochromatic subtree of $S$ for $g$ is also monochromatic for $f$.
\end{lemma}

\begin{proof}
Working in $\rca$, suppose $R$, $f$, and $A$ are as in the statement of the lemma.
We will essentially carry out the proof of Theorem 1.4 of \cite{chm}, replacing uses of
arithmetical comprehension by recursive comprehension relative to $A^{\prime\prime}$.
Toward this end, given a sequence $P = \{ \rho _ \tau \mid \tau \subseteq \sigma\}$
of comparable nodes of $R$ such that the sequence terminates in $\rho_\sigma$,
define an induced coloring of single nodes $\tau \supset \rho_\sigma$ by setting
\[
f_{\rho_\sigma} (\tau ) = \langle \{ ( \vec m , f (\vec m , \tau ) ) \mid \vec m \in [P]^n \} \rangle
\]
where the angle brackets denote an integer code for the finite set.
Since $f \le_T A$, for any finite set $P$ we have $f_{\rho_\sigma } \le_T A$.

For each $\sigma \in \tre$, define $p_\sigma$, $T_\sigma$, and $c_\sigma$ as follows.
Let $\rho_{\langle\rangle}$ be the root of $R$ and $T_{\langle\rangle} = R$.  Given $\rho_\sigma$
and $T_\sigma$ computable from $A$, use $A^{\prime\prime}$ to compute a $c_\sigma$
which is the greatest integer in the range of $f_{\rho_\sigma}$ such that
\[
\exists \rho \in T_\sigma ( \rho \supset \rho_\sigma \land \forall \tau \in T_\sigma
(\tau \supset \rho \to c_\sigma \le f_{\rho_\sigma} (\tau ))) .
\]
Using $A^{\prime\prime}$, compute the least such $\rho$.  Let $T$ denote the subtree of $T_\sigma$
isomorphic to $\tre$ defined by taking $\rho$ as the root and letting the immediate
successors of each node be the least pair of incomparable extensions in $T_\sigma$ that are
assigned $c_\sigma$ by $f_{\rho_\sigma}$.  Because of the choice of $c_\sigma$, $T$ is isomorphic
to $\tre$, and its nodes can be located in an effective manner.
(In \cite{chm}, this $T$ is called the standard $c_\sigma$-colored
subtree of $T_\sigma$ for $\rho$ using $f_{\rho_\sigma}$.)  Let $\rho_{\sigma\cat 0}$ and
$\rho_{\sigma\cat 1}$ be the two level one elements of $T$ and let $T_{\sigma\cat \varepsilon}$
be the subtree of $T$ with root $\rho_{\sigma\cat \varepsilon}$ for each $\varepsilon \in \{0,1\}$.
Note that given any finite chain of elements and colors
$\{\langle \rho_\tau , c_\tau \rangle \mid \tau \subseteq \sigma \}$, sufficiently large initial
segments of each $T_\tau$ can be computed to determine $\rho_{\sigma\cat 0}$,
$\rho_{\sigma\cat 1}$, $c_{\sigma\cat 0}$, and $c_{\sigma\cat 1}$, using only
$A ^{\prime\prime}$.  Consequently, the subtree $S = \{ \rho_\sigma \mid \sigma \in \tre \}$ is computable from
$A^{\prime\prime}$.

Define $g:[S]^n \to k$ by
$g(p_{\sigma_1} , \dots , \rho_{\sigma_n} ) = f( \rho_{\sigma _1} , \dots , \rho_{\sigma_n } , \rho_{\sigma_n \cat 0} )$.
Since $S \le_T A^{\prime\prime}$, we also have $g \le _T A ^{\prime\prime}$.  By the construction of $S$,
given any increasing sequence of elements of $S$ of the form
$\rho _1 \subset \rho_2 \subset \dots \subset \rho_n $,
and extensions $\rho_n \subset \rho_{n+1}$ and $\rho_n \subset \rho_{n+2}$,
we have $f_{\rho_n} (\rho_{n+1} ) = f_{\rho_n} (\rho_{n+2} )$, so
$f(\rho_1 , \dots , \rho_n , \rho_{n+1} ) = f( \rho_1 , \dots , \rho_n , \rho_{n+2} )$.
Thus any monochromatic subtree for $g$ is also monochromatic for $f$, and the proof is complete.
\end{proof}

Extracting the 
computability theoretic content of the previous argument, given
a computable coloring of $n$-tuples we can find a monochromatic
set computable from $0^{(2n-2)}$.  This is not an optimal bound,
since applying the Strong Hierarchy Theorem to Theorem 2.7 of \cite{chm}
yields a monochromatic set computable from $0^{(n)}$.  However, the
preceding result does enable us to complete the proof of the next theorem,
and avoids formalization of the long proof of Theorem 2.7 of \cite{chm}.

\begin{thm}\label{theoremB}
$(\rca )$  The following are equivalent:
\begin{list}{}{}
\item [$(1)$]  $\acap$
\item [$(2)$]  $\forall n \ctt (n)$
\end{list}
\end{thm}

\begin{proof}
To prove that (1) implies (2), assume $\acap$ and let $f: [\tre ]^n \to k$ be a coloring.
By $\acap$, the jump $f^{(2n-2)}$ exists, so by discarding the odd jumps we can
find a sequence of sets $X_0 , X_1 , \dots , X_{n-1}$ such that $X_0 = f$ and for each $i$,
$X_{i+1} = X_i^{\prime\prime}$.
Note that $f \le_T X_0$ and $\tre \le_T X_0$.
By Lemma \ref{lemmaA}, for any $X_i$, given indices
witnessing that a subtree isomorphic to $\tre$ and a coloring of the $(n-i)$-tuples
of that subtree are each computable from $X_i$, we can find indices for computing
an infinite subtree and a coloring of $(n-i-1)$-tuples from $X_{i+1}$ satisfying the
conclusion of Lemma~\ref{lemmaA}.  Thus, by
induction on arithmetical formulas (which is a consequence of $\acap$),
we can prove the existence of a sequence of indices, the last of which
can be used to compute a subtree $T_{n-1}$ and a function $f_{n-1}:[T_{n-1}] ^1 \to k$ such
that
$T_{n-1}$ is isomorphic to $\tre$ and
any monochromatic
subtree for $f_{n-1}$ is also monochromatic for $f$.  Since $\acap$ includes $\rca$ plus
induction for $\Sigma^0_2$ formulas, by Theorem 1.2 of \cite{chm}, $T_{n-1}$ contains
a subtree which is monochromatic for $f_{n-1}$ and isomorphic to $\tre$.  This subtree
is also monochromatic for $f$, so $\ctt (n)$ holds for $f$.

To prove that (2) implies (1), assume $\rca$ and (2).  Given any coloring of $n$-tuples of integers,
$f:[\nat ]^n \to k$, we may define a coloring $g: [ 2^{<\nat } ] ^n \to k$ on $n$-tuples of elements
of $\tre$ by
\[
g(\sigma_1 , \dots , \sigma_n ) = f( \lh (\sigma_1 ) , \dots \lh (\sigma_n))
\]
where $\lh (\sigma)$ denotes the length of the sequence $\sigma$.
Any monochromatic tree for $g$ contains an infinite path which encodes an infinite
monochromatic set for $f$.  Thus, as noted in the proof of Theorem 1.5 of \cite{chm},
$\forall n \ctt (n)$ implies the usual full Ramsey's theorem, denoted by $\forall n {\sf{RT}}(n)$.
$\acap$ can be deduced from $\forall n {\sf{RT}} (n)$ by Theorem 8.4 of
\cite{mileti}, or by applying Proposition 4.4 of \cite{dh}.
\end{proof}

A typical proof of $\forall n \ctt (n)$ would proceed by induction on $n$ and require the use of induction
on $\Pi^1_2$ formulas.  In the preceding argument, the existence of the $n$th jump is used to push
the application of induction down to arithmetical formulas.  The proof of Theorem \ref{theoremB} together
with Proposition 4.4 of \cite{dh} provide a detailed exposition of a proof and reversal in $\acap$ and
 show that the full versions of the usual Ramsey's theorem, the polarized version of
Ramsey's theorem, and Ramsey's theorem for trees are all equivalent to $\acap$ over $\rca$.

\end{document}